\documentclass[11pt]{amsart}

\usepackage[centertags]{amsmath}
\usepackage{amsfonts}
\usepackage{amssymb}
\usepackage{amsthm}
\usepackage[english]{babel}
\usepackage[active]{srcltx}

\usepackage[colorlinks]{hyperref}

\newcommand{\Real}{\mathbb{R}}
\newcommand{\Natural}{\mathbb{N}}

\newcommand{\Lip}{\operatorname{Lip}}
\newcommand{\supp}{\operatorname{supp}}

\newtheorem{thm}{Theorem}[section]
\newtheorem{cor}[thm]{Corollary}
\newtheorem{lem}[thm]{Lemma}
\newtheorem{prop}[thm]{Proposition}

\newtheorem{rem}[thm]{Remark}

\newtheorem{defn}[thm]{Definition}

\numberwithin{equation}{section}

\newenvironment{prooftheoMyersNakai}%
    {\par\noindent{\it Proof of Theorem \ref{Myers:Nakai}. }\nopagebreak\normalsize}%
    {\hfill\linebreak[2]\hspace*{\fill}$\square$\\ }

\begin{document}

\title[Algebras of smooth functions on Banach-Finsler manifolds]{Characterization of a Banach-Finsler manifold in terms of the algebras of smooth functions}
%{Algebras of smooth functions on Banach-Finsler manifolds}
%{A Myers-Nakai theorem on Banach-Finsler manifolds}

\author{J.A.  Jaramillo, M. Jim{\'e}nez-Sevilla and L. S\'anchez-Gonz\'alez}

\address{Departamento de An{\'a}lisis Matem{\'a}tico\\ Facultad de
Matem{\'a}ticas\\ Universidad Complutense de Madrid\\ 28040 Madrid, Spain}

\thanks{Supported in part by DGES (Spain) Project MTM2009-07848. L. S\'anchez-Gonz\'alez has  also been supported by grant MEC AP2007-00868}

\email{jaramil@mat.ucm.es, marjim@mat.ucm.es,
lfsanche@mat.ucm.es}



\date{August, 2011}

%%%%%%
\maketitle

%%%%%%%%%

\begin{abstract}
In this note we give sufficient conditions to ensure that the weak Finsler structure of a complete $C^{k}$ Finsler manifold $M$ is determined by the normed algebra $C_b^k(M)$  of all real-valued, bounded and $C^k$ smooth functions with bounded derivative defined on $M$. As a consequence, we obtain: (i) the  Finsler structure of a finite-dimensional and complete $C^{k}$ Finsler manifold  $M$ is  determined by the algebra $C_b^k(M)$;  (ii) the weak Finsler structure of a separable and complete $C^{k}$ Finsler manifold $M$ modeled on a Banach space with a Lipschitz and $C^k$ smooth bump function is determined by the algebra $C^k_b(M)$; (iii) the weak Finsler structure of a $C^k$ uniformly bumpable and complete $C^{k}$ Finsler manifold $M$ modeled on a Weakly Compactly Generated (WCG) Banach space with an (equivalent)  $C^k$ smooth norm  is determined by the algebra $C^k_b(M)$;
and (iv) the isometric structure of a WCG Banach  space $X$ with an $C^1$ smooth bump function is determined by the algebra $C_b^1(X)$.
\end{abstract}

%%%%%%%%%%%%%%

\section{Introduction and Preliminaries}
%%%%%%%%%%%%
%%%%%%
In this note, we are interested in characterizing the Finsler structure of a Finsler manifold $M$  in terms of the space of real-valued, bounded and $C^k$ smooth functions with bounded derivative defined on $M$. The problem of the interrelation of  the topological, metric and smooth structure of a space $X$  and  the algebraic and topological  structure of the space $C(X)$ (the set of real-valued continuous functions defined on $X$)  has been largely studied. These results are usually referred as \emph{Banach-Stone type theorems}.  Recall the celebrated Banach-Stone theorem, asserting that the compact spaces $K$ and $L$ are homeomorphic if
 and only if the Banach spaces $C(K)$ and $C(L)$ endowed with the sup-norm are isometric. For more information on Banach-Stone type theorems see the survey \cite{GaJaExtracta} and references therein.

%the topology of a compact space $K$ is determined by the linear and metric structure of  the Banach space $C(K)$ endowed with the sup-norm (for more information about Banach-Stone type theorems see the survey \cite{GaJaExtracta} and references therein).

The Myers-Nakai theorem states that the structure of a complete Riemannian manifold $M$ is characterized in terms of the Banach algebra $C^1_b(M)$ of all real-valued, bounded and $C^1$ smooth functions  with bounded derivative defined on $M$ endowed with the sup-norm of the function and its derivative.  More specifically, two complete Riemannian manifolds $M$ and $N$ are equivalent as Riemannian manifolds, i.e. there is a $C^1$ diffeomorphism $h:M\to N$ such that
\begin{equation*}
\langle dh(x)(v),dh(x)(w)\rangle_{h(x)}=\langle v,w \rangle_x
\end{equation*}
for every $x\in M$ and $v,w\in T_xM$ if and only if the Banach algebras $C^1_b(M)$ and $C^1_b(N)$  are isometric. This result was first proved by S. B. Myers \cite{Myers} for a compact and Riemannian manifold and by M. Nakai \cite{Nakai} for a finite-dimensional  Riemannian manifold. Very recently,  I. Garrido, J.A. Jaramillo and Y.C. Rangel \cite{GaJaRa} gave an extension of the  Myers-Nakai theorem for every infinite-dimensional, complete Riemannian manifold. A similar result  for  the so-called finite-dimensional Riemannian-Finsler  manifolds is given in \cite{GaJaRa2} (see also \cite{Rangel-Tesis}).
% An interesting result for  Banach manifolds was  obtained in \cite{GAJAANG}, where it was proved that the smooth structure of a $C^\infty$ smooth Banach manifold $M$ is determined by the algebra $C^{\infty}(M)$ (the space of all real-valued and $C^\infty$ smooth functions defined on $M$) whenever $M$ is modeled on a Banach space with a $C^\infty$ smooth bump function.

 Our aim in this work  is to extend the  Myers-Nakai theorem to the context of Finsler manifolds. On the one hand, we obtain  the Myers-Nakai theorem  for (i) finite-dimensional and complete Finsler manifolds, and (ii) WCG Banach  spaces  with a $C^1$ smooth bump function. On the other hand, we study for $k\ge 1$ the algebra $C^k_b(M)$ of all real-valued, bounded and $C^k$ smooth functions with  bounded first derivative defined on a complete Finsler manifold $M$. We prove that these algebras determine 
the weak Finsler structure of a complete Finsler manifold when $k=1$ and the Finsler structure when $k\ge 2$. 
 In particular, we obtain  a weaker
  version of the Myers-Nakai theorem for  (i) separable and complete Finsler manifolds modeled on a Banach space with a Lipschitz and $C^k$ smooth bump function,  and (ii) $C^k$ uniformly bumpable and complete Finsler manifolds  modeled on  WCG Banach spaces with an equivalent $C^k$ smooth norm. In the proof of these results we will use the ideas of the Riemannian case \cite{GaJaRa}.

    \smallskip

The notation we use is standard. The norm in a Banach space $X$ is denoted by $||\cdot||$.  The dual space of $X$ is denoted by
$X^*$ and its dual norm by $||\cdot||^*$. The open ball with center $x\in X$ and radius $r>0$ is denoted by $B(x,r)$. A $C^k$ smooth bump function $b:X\to \Real$ is a $C^k$ smooth function on $X$ with bounded, non-empty support, where $\supp(b) =\overline{\{x \in  X : b(x)\neq 0\}}$.
If $M$ is a Banach manifold, we denote by $T_xM$ the tangent space of $M$ at $x$. Recall that the tangent bundle of $M$ is $TM=\{(x,v):x\in M \text{ and } v\in T_x M\}$. We refer to \cite{DGZ}, \cite{fabianhajek}, \cite{Lang} and \cite{Deim}  for additional definitions. We will say that the norms $||\cdot||_{1}$
and   $||\cdot||_{2}$ defined on a Banach space $X$
 are $K$-equivalent ($K\ge 1$) whether $\frac{1}{K}||v||_1\le||v||_2\le K||v||_1$, for every
$v\in X$.

\smallskip

Let us begin by recalling the definition of  a  $C^k$ Finsler manifold  in the sense of Palais as well as some basic properties (for more information about these manifolds see  \cite{Palais},  \cite{Deim}, \cite{Rabier}, \cite{Neeb}, \cite{GaGuJa} and  \cite{MarLuis}).

\begin{defn} \label{defFinsler}
Let $M$ be a (paracompact) $C^k$ Banach manifold modeled on a Banach space $(X,||\cdot||)$, where $k \in\Natural\cup \{\infty\}$. Let us consider the   tangent  bundle $TM$
of $M$ and  a continuous  map $||\cdot||_M: TM\to [0,\infty)$. We say that   $(M,||\cdot||_M)$ is a \textbf{$C^k$ Finsler manifold in the sense of Palais} if $||\cdot||_M$ satisfies the following conditions:
\begin{enumerate}
\item[(P1)] For every $x\in M$, the map $||\cdot||_x:={||\cdot||_M}_{\mid_{T_xM}}:T_xM\to [0,\infty)$ is a norm on the tangent space $T_xM$ such that for every chart $\varphi:U\to X$ with $x \in U$, the norm $v\in X \mapsto ||d\varphi^{-1}(\varphi(x))(v)||_x$ is equivalent to $||\cdot||$ on $X$.
\item[(P2)] For every $x_0\in M$, every $\varepsilon>0$ and every chart $\varphi:U\to X$ with $x_0\in U$, there is an open neighborhood $W$ of $x_0$  such that if  $x\in W$ and $v\in X$, then
\begin{equation}\label{palaisdef}
\frac{1}{1+\varepsilon}||d\varphi^{-1}(\varphi(x_0))(v)||_{x_0}\le ||d\varphi^{-1}(\varphi(x))(v)||_{x}\le (1+\varepsilon)||d\varphi^{-1}(\varphi(x_0))(v)||_{x_0}.
\end{equation}
In terms of equivalence of norms, the above inequalities yield the fact that the norms $||d\varphi^{-1}(\varphi(x))(\cdot)||_{x}$
and   $||d\varphi^{-1}(\varphi(x_0))(\cdot)||_{x_0}$ are $(1+\varepsilon)$-equivalent.
\end{enumerate}
\end{defn}

Let us recall  that  Banach spaces and Riemannian manifolds   are  $C^\infty$ Finsler manifolds in the sense of Palais \cite{Palais}.

Let $M$ be a  Finsler manifold, we denote by  $T_xM^*$ the dual space of the tangent space $T_xM$. Let $f:M\to\Real$ be a differentiable function at $p\in M$. The norm of $df(p)\in {T_p M}^*$ is given by
%%%
\begin{equation*}
||df(p)||_p= \sup\{|df(p)(v)|:v\in T_p M, ||v||_p\le 1\}.
\end{equation*}
%%%
Let us consider a differentiable function  $f:M\to N$  between Finsler manifolds $M$ and $N$. The norm of the derivative at the point $p\in M$ is defined as
%%%
\begin{align*}
& ||df(p)||_p= \sup\{||df(p)(v)||_{f(p)}:v\in T_p M, ||v||_p\le 1\} = &\\
& \qquad \qquad =\sup\{\xi(df(p)(v)):\xi\in {T_{f(p)} N}^*,\ v\in T_pM\ \text{and}\ ||v||_{p}=1=||\xi||^*_{f(p)} \},&
\end{align*}
%%%
where $||\cdot||_{f(p)}^*$ is the dual norm  of $||\cdot||_{f(p)}$. Recall that if $(M, ||\cdot||_M)$ is a
 Finsler manifold, the {\em length} of a piecewise $C^1$ smooth
path $c:[a,b]\rightarrow M$ is defined as $\ell(c):=\int_{a}^b||c'(t)||_{c(t)}\,dt$. Besides, if $M$ is connected, then it is
connected by piecewise $C^1$ smooth paths, and the associated
{\em Finsler metric} $d_M$ on $M$ is defined as
%%%
\begin{equation*}
d_M(p,q)=\inf\{\ell(c): \, c \text{ is a piecewise } C^1 \text{ smooth path connecting } p \text{ and } q\}.
\end{equation*}
%%%
It was shown in \cite{Palais} that the Finsler metric is consistent with the topology given in $M$. The open ball of center $p\in M$ and radius
$r>0$ is  denoted by $B_M(p,r):=\{q\in M:\, d_M(p,q)<r\}$. The Lipschitz constant $\Lip(f)$ of a Lipschitz function $f:M\rightarrow N$, where
$M$ and $N$ are Finsler manifolds, is defined as $\Lip(f)=\sup\{\frac{d_N(f(x),f(y))}{d_M(x,y)}: x,y \in M, x\not=y\}$. We shall only consider connected manifolds. Let us recall the following ``mean value inequality"  for Finsler manifolds \cite{AzFeMe, MarLuis}.

 \begin{lem}  \label{meanvalue} Let $M$ and $N$ be $C^1$ Finsler manifolds (in the sense of Palais)
and $f:M\to N$  a  $C^1$ smooth  function. Then,  $f$ is Lipschitz if and only if  $||df||_\infty := \sup\{||df(x)||_x:x\in M\}<\infty$. Furthermore, $\Lip(f)= ||df||_\infty$.
\end{lem}
%%%%

We will also need the following result related to the  $(1+\varepsilon)$-bi-Lipschitz local behavior of the charts of a $C^1$ Finsler manifold in the sense of Palais  \cite[Lemma 2.4]{MarLuis}.

\begin{lem} \label{desigualdades:BiLipschitz}
%%%
Let us consider a  $C^1$ Finsler manifold $M$ (in the sense of Palais). Then, for every $x_0\in M$ and every chart
$(U,\varphi)$  with $x_0\in U$ satisfying inequality \eqref{palaisdef}, there exists an open neighborhood
 $V\subset U$ of $x_0$ satisfying
\begin{equation}\label{B-L1}
\frac{1}{1+\varepsilon}d_M(p,q)\le |||\varphi(p)-\varphi(q)|||\le (1+\varepsilon) d_M(p,q), \quad \text{ for every } p,q\in V,
\end{equation}
where   $|||\cdot|||$ is the (equivalent) norm $||d\varphi^{-1}(\varphi(x_0))(\cdot)||_{x_0}$ defined  on $X$.
\end{lem}

%%%%%%%%%

Now, let us  recall the concept of {\em uniformly bumpable manifold},  introduced by D. Azagra, J. Ferrera and F. L\'opez-Mesas \cite{AzFeMe} for Riemannian manifolds.  A natural extension to {Finsler manifolds} is defined in the same way \cite{MarLuis}.
%%%%
\begin{defn} \label{bumpable}
A  $C^k$  Finsler manifold in the sense of Palais $M$  is \textbf{$C^k$ uniformly bumpable} whenever there are $R>1$ and $r>0$ such that for every $p\in M$ and $\delta\in (0,r)$ there exists a $C^k$ smooth function $b:M\to[0,1]$ such that:
\begin{enumerate}
\item $b(p)=1$,
\item $b(q)=0$ whenever $d_M(p,q)\ge \delta$,
\item $\sup_{q\in M} ||d b(q)||_q \le R/\delta$.
\end{enumerate}
\end{defn}
%%%%
Note that this is not a restrictive definition:  D. Azagra,  J.  Ferrera,   F. L\'opez-Mesas and  Y. Rangel \cite{AzFeMeRa} proved that  every separable Riemannian manifold is $C^\infty$ uniformly bumpable. This result was generalized  in \cite{MarLuis}, where it was proved that  every $C^1$ Finsler manifold (in the sense of Palais) modeled on a certain class of Banach spaces (such as Hilbert spaces, Banach spaces with separable dual, closed subspaces of $c_0(\Gamma)$ for every set $\Gamma\neq \emptyset$) is $C^1$ uniformly bumpable. In particular, every  Riemannian manifold (either separable or non-separable)  is  $C^1$ uniformly bumpable.

It is straightforward to verify that if a  $C^{k}$ Finsler manifold $M$ is  modeled on a Banach space $X$ and $M$ is  $C^k$ uniformly bumpable, then $X$ admits  a Lipschitz $C^k$ smooth bump function. Besides, a {\em separable} $C^k$ Finsler manifold $M$ is modeled on a Banach space with  a Lipschitz, $C^k$ smooth bump function if and only if $M$ is $C^k$ uniformly bumpable \cite{MarLuis}. Nevertheless, we do not know whether this equivalence  holds in the non-separable case.

%It is straightforward to verify that a  $C^{k}$ Finsler manifold $M$ is  modeled on a Banach space with a Lipschitz $C^k$ smooth bump function whenever $M$ is $C^k$ uniformly bumpable. Besides, a {\em separable} $C^k$ Finsler manifold $M$ is modeled on a Banach space with  a Lipschitz, $C^k$ smooth bump function if and only if $M$ is $C^k$ uniformly bumpable \cite{MarLuis}. Nevertheless, we do not know whether this assertion  holds in the non-separable case.

%%%%
From now on, we shall refer to  $C^{k}$ Finsler manifolds  in the sense of Palais as $C^{k}$ Finsler manifolds, and $k \in \Natural \cup \{\infty\}$. 
We shall use the standard notation of  $C^k(U,Y)$  for the set of  all $k$-times continuously differentiable functions defined on an open subset $U$ of a Banach space (Finsler manifold) taking values into a Banach space (Finsler manifold) $Y$. We shall write $C^k(U)$ whenever $Y=\mathbb R$.

Now, let us recall the concept of weakly  $C^k$ smooth function.

\begin{defn} Let $X$ and $Y$ be  Banach spaces and  consider a function  $f:U\to Y$, where  $U$ is an open subset of $X$. The function
 $f$ is said to be \textbf{weakly  $C^k$ smooth} at the point $x_0$ whenever there is an open neighborhood $U_{x_0}$ of $x_0$  such that $y^*\circ  f$  is $C^k$ smooth at $U_{x_0}$, for every $y^*\in Y^*$. The   function $f$ is said to be \textbf{weakly $C^k$ smooth} on $U$ whenever $f$ is weakly $C^k$ smooth at  every point $x\in U$.
 \end{defn}

On the one hand,  J. M. Guti\'errez and J.L. G. Llavona   \cite{GULLA} proved that  if  $f:U\to Y$ is weakly $C^k$ smooth on $U$, then   $g\circ f\in C^k(U)$ for all $g\in C^k(Y)$. They also proved that   if  $f:U\to Y$ is weakly $C^k$ smooth on $U$, then $f\in C^{k-1}(U)$. For $k=1$, the above yields that every weakly $C^1$ smooth function on $U$ is continuous on $U$. Also, for $k=\infty$, every weakly $C^\infty$ smooth function on $U$ is $C^\infty$ smooth on $U$.  M. Bachir and  G. Lancien \cite{BaLan} proved that, if the Banach space $Y$ has  the Schur property, then the concept of
weakly $C^k$ smoothness coincides with  the concept of $C^k$ smoothness.
On the other hand, there are examples of  weakly $C^1$ smooth functions that are not $C^1$ smooth (see \cite{GULLA} and \cite{BaLan}).

\begin{defn}
Let $M$ and $N$  be $C^{k}$  Finsler manifolds and  $U\subset M$, $O\subset N$   open subsets of $M$ and $N$, respectively.
A function $f:U\to N$ is said to be \textbf{weakly $C^k$ smooth} at the point $x_0$ of $U$ if there exist charts $(W,\varphi)$ of $M$ at  $x_0$ and $(V,\psi)$ of $N$ at $f(x_0)$ such that $ \psi\circ f \circ \varphi^{-1}$ is weakly $C^k$ smooth at $\varphi(W)$. We say that $f:U\to N$ is  \textbf{weakly $C^k$ smooth}  in $U$ if $f$ is weakly $C^k$ smooth at every point $x\in U$. 
We say that a bijection $f:U\rightarrow O$ is a weakly $C^k$ diffeomorphism if $f$ and $f^{-1}$ are weakly $C^k$ smooth on $U$ and $O$, respectively.
Notice that these definitions do not depend on the chosen charts.
\end{defn}
%%%%%%%%%%
%%%%%%%%%%%%%

Let us note that there are homeomorphisms which are weakly $C^1$ smooth but not differentiable. Indeed, we follow \cite[Example 3.9]{GULLA} 
and define $g:\Real \to c_0(\Natural)$ and  $h:c_0(\Natural)\to c_0(\Natural)$ by $g(t)=(0,\frac{1}{2}\sin(2t),\dots,\frac{1}{n}\sin(nt),\dots)$ and $h(x)=x+g(x_1)$ for every $t\in \Real$ and $x=(x_1,\dots, x_n, \dots)\in c_0$.  The function $h$ is an homeomorphism, $h^{-1}(y)=y-g(y_1)$ for every $y\in c_0$, and $h$ is weakly $C^1$ smooth on $c_0(\Natural)$. Notice that if $h$ were differentiable
at a point  $x\in c_0$ with $x_1=0$, then
\begin{equation*}
h'(x)(1,0,0, \dots)=(1,1,1,\dots)Ê\in \ell_\infty \setminus c_0,
\end{equation*}
which is a contradiction.

%%%%%%%%%%%%%%
%%%%%%%%%%%%
Now, let us consider different definitions of isometries between $C^k$ Finsler manifolds.

\begin{defn}
Let $(M,||\cdot||_M)$ and $(N,||\cdot||_N)$ be $C^{k}$ Finsler manifolds and
a bijection  $h:M\to N$.
\begin{itemize}
\item[({\bf{MI}})] We say that  $h$ is a \textbf{metric isometry} for the Finsler metrics,  if
\begin{equation*}\label{metric-isometry}
d_N(h(x),h(y))=d_M(x,y), \qquad \text{ for every } x,y \in M.
\end{equation*}
\item[{\bf(FI)}]  We say that $h$ is a \textbf{$C^{k}$  Finsler isometry} if it is a $C^{k}$ diffeomorphism satisfying
\begin{equation*}
||dh(x)(v)||_{h(x)}=||(h(x),dh(x)(v))||_N=||(x,v)||_M=||v||_x,
\end{equation*}
for every $x\in M$ and $v\in T_x M$.  We  say that the Finsler manifolds $M$ and $N$ are \textbf{$C^{k}$ equivalent as Finsler manifolds} if there is a  $C^{k}$ Finsler isometry between $M$ and $N$.
%%%%
\item[({\bf{$\omega$-FI}})] We say that  $h$ is a \textbf{weak $C^{k}$ Finsler isometry} if it is a weakly $C^k$   diffeomorphism and a metric isometry for the Finsler metrics.
We  say that the Finsler manifolds $M$ and $N$ are \textbf{weakly $C^k$ equivalent as Finsler manifolds} if there is a  weak $C^{k}$ Finsler isometry between $M$ and $N$.
\end{itemize}
\end{defn}

\begin{prop}\label{Finsler:Isometry}
Let $M$  and $N$ be $C^{k}$ Finsler manifolds. Let us assume that there is a  $C^k$ diffeomorphism and metric isometry
(for the Finsler metrics)  $h:M\to N$. Then $h$ is a $C^k$  Finsler isometry.
\end{prop}

\begin{proof}
Let us fix $x\in M$ and $y=h(x)\in N$. For every $\varepsilon>0$, there are $r>0$ and charts $\varphi:B_M(x,r)\subset M \to X$ and  $\psi: B_N(y,r) \subset N \to Y$ satisfying  inequalities  \eqref{palaisdef} and \eqref{B-L1}. Since $h:M \to N$ is a metric isometry,  $h$ is a bijection from $B_M(x,r)$ onto $B_N(y,r)$.

Let us consider the equivalent norms on $X$ and $Y$ defined as
$
|||\cdot|||_x:=||d\varphi^{-1}(\varphi(x))(\cdot)||_x $ and $|||\cdot|||_y=||d\psi^{-1}(\psi(y))(\cdot)||_y,$  respectively.

Since $h$ is a metric isometry, we obtain from Lemma \ref{desigualdades:BiLipschitz},  for $p,q$ in an open neighborhood
of $\varphi(x)$,
%%%%%%%%%%
\begin{align*}
|||\psi\circ h\circ \varphi^{-1}(p)-\psi\circ h\circ  \varphi^{-1}(q)|||_y  \le  (1+\varepsilon)d_N(h\circ \varphi^{-1}(p),  h\circ \varphi^{-1}(q)) = &\\
 = (1+\varepsilon)d_M(\varphi^{-1}(p),\varphi^{-1}(q))\le (1+\varepsilon)^2|||p-q|||_x. &
\end{align*}
%%%%%%%%%%%
Thus, $\sup\{|||d(\psi\circ h\circ \varphi^{-1})(\varphi(x))(w)|||_y: |||w|||_x\le 1 \}\le (1+\varepsilon)^2$.
Now, for every $v\in T_x M$ with $v\neq 0$, let us write $w=d\varphi(x)(v)\in X$. We have
%%%%%%%%%%%%
\begin{align*}
||dh(x)(v)||_y & =||d\psi^{-1}(\psi(y))d\psi(y)dh(x)(v)||_y=|||d (\psi\circ h) (x)(v)|||_y = \\
& = ||| d (\psi\circ h) (x) d\varphi^{-1}(\varphi(x))(w)|||_y
  = ||| d(\psi\circ h \circ \varphi^{-1}) (\varphi(x))(w)|||_y \le \\
  & \le (1+\varepsilon)^2 |||w|||_x=(1+\varepsilon)^2 ||v||_x.
\end{align*}
%%%%%%%%%%%%%

Since this inequality holds for every $\varepsilon>0$ and the same argument works for $h^{-1}$, we conclude that $||dh(x)(v)||_y=||v||_x$ for all $v\in T_xM$. Thus,  $h$ is a $C^k$ Finsler isometry.
\end{proof}

%%%%%%%%%%%%%%%%%%%%%%%%

Let us now turn  our attention to the {\em Banach algebra} $C_b^1(M)$, the algebra of all real-valued,  $C^1$ smooth and bounded functions with bounded derivative defined on a $C^1$ Finsler manifold $M$, i.e.
\begin{equation*}
C_b^1(M)=\{f:M\to\Real: f\in C^1(M),\ ||f||_\infty<\infty \text{ and } ||df||_\infty<\infty\},
\end{equation*}
where $ ||f||_\infty:=\sup\{|f(x)|:\, x\in M\}$ and $||df||_\infty:=\sup\{||df(x)||_x:\, x\in M\}$.
The usual norm considered on  $C_b^1(M)$ is $||f||_{C_b^1}=\max\{||f||_\infty,||df||_\infty\}$ for every $f\in C_b^1(M)$
and $(C_b^1(M),||\cdot||_{C_b^1(M)})$ is a Banach space. Let us notice that, by Lemma \ref{meanvalue}, we have $||df||_\infty = \Lip(f)$. Recall that $(C_b^1(M),2||\cdot||_{C_b^1(M)})$ is a Banach algebra.

For   $2\le k \le  \infty$ and a $C^{k}$ Finsler manifold $M$, let us  consider the algebra  $C_b^k(M)$ of all real-valued,  $C^k$ smooth and bounded functions that have bounded first derivative, i.e.
\begin{equation*}
C_b^k(M)=\{f:M\to\Real: f\in C^k(M),\ ||f||_\infty<\infty \text{ and } ||df||_\infty<\infty\}=C^k(M)\cap C_b^1(M).
\end{equation*}
with the norm $||\cdot||_{C_b^1}$. Thus,  $C_b^k(M)$ is a subalgebra of $C_b^1(M)$. Nevertheless, it is not a Banach algebra.
 %We can follow the steps given in \cite{GaJaRa} to establish   the properties of  the algebra $C_b^k(M)$, in an similar way as in
 %\cite{Isbell} and \cite{GaJa}.

A function   $\varphi:C_b^k(M)\to \Real$ ($1\le k\le \infty$) is said to be an \emph{algebra homomorphism} whether for  all $f,g\in C_b^k(M)$ and $\lambda,\eta \in \Real$,
\begin{enumerate}
\item[(i)] $\varphi(\lambda f + \eta g)=\lambda \varphi(f)+ \eta \varphi(g)$, and
\item[(ii)] $\varphi(f\cdot g)=\varphi(f)\varphi(g)$.
\end{enumerate}

Let us denote  by  $H(C_b^k(M))$  the set of all nonzero algebra homomorphisms, i.e.
\begin{equation*}
H(C_b^k(M))=\{\varphi: C_b^k(M)\to \Real: \varphi \text{ is an algebra homomorphism and } \varphi(1)=1\}.
\end{equation*}

Let us list some of the basic properties of the algebra $C_b^k(M)$ and the algebra homomorphisms $H(C_b^k(M))$.
They can be checked as in the Riemannian case (see \cite{GaJa},  \cite{GaJaRa} and \cite{Isbell}).

\begin{itemize}
\item[(a)] If $\varphi\in H(C_b^k(M))$, then  $\varphi\not=0$ if  and only if  $\varphi(1)=1$.
\item[(b)] If $\varphi\in H(C_b^k(M))$, then $\varphi$ is positive, i.e. $\varphi(f)\ge 0$ for every $f\ge 0$.
\item[(c)] If the $C^{k}$ Finsler manifold $M$ is modeled on a Banach space that admits a Lipschitz and $C^k$ smooth bump function, then $C_b^k(M)$ is a {\em unital algebra that separates points and closed sets} of $M$.
      Let us briefly give the proof  for completeness. Let us take  $x\in M$, and $C\subset M$ a closed subset of $M$ with $x\not\in C$. Let us take $r>0$ small enough so that $C\cap  B_M(x,r) =\emptyset$ and  a chart  $\varphi:B_M(x,r)\to X$  satisfying inequality  \eqref{palaisdef}. Let us take  $s>0$ small enough so that $\varphi(x)\in B(\varphi(x),s)\subset \varphi(B(x,r/2))\subset X$ and a  Lipschitz and $C^k$ smooth bump function $b:X\to\Real$ with $b(\varphi(x))=1$ and $b(z)=0$ for every $z\not\in B(\varphi(x),s)$. Let us define $h:M\to \Real$ as $h(p)=b(\varphi(p))$ for every $p\in B_M(x,r)$ and $h(p)=0$ otherwise. Then $h\in C_b^k(M)$, $h
 (x)=1$ and $h(c)=0$ for every $c\in C$.
\item[(d)] The space $H(C_b^k(M))$ is closed as a topological subspace of $\Real^{C_b^k(M)}$ with the product topology.
 Moreover, since every function in $C_b^k(M)$ is bounded, it can be checked
that  $H(C_b^k(M))$ is compact in $\Real^{C_b^k(M)}$.
\item[(e)] If  $C_b^k(M)$ separates points and closed subsets, then $M$ can be embedded as a topological subspace of $H(C_b^k(M))$ by identifying every $x\in M$  with the \emph{point evaluation homomorphism}  $\delta_x$ given by
$\delta_x(f)=f(x)$ for every $f\in C_b^k(M)$.  Also, it can be checked that  the subset  $\delta(M)=\{\delta_x:x\in M\}$ is dense in $H(C_b^k(M))$. Therefore, it follows that $H(C_b^k(M))$ is a compactification of  $M$.
    \item[(f)] Every $f\in C_b^k(M)$ admits a continuous extension $\widehat{f}$ to  $H(C_b^k(M))$, where $\widehat{f}(\varphi)=\varphi(f)$ for every $\varphi\in H(C_b^k(M))$.  Notice that this extension $\widehat{f}$ coincides in $H(C_b^k(M))$ with the projection $\pi_f:\Real^{C_b^k(M)}\to \Real$, given by $\pi_f(\varphi)=\varphi(f)$, i.e. ${\pi_f}_{\mid_{H(C_b^k(M))}}=\widehat{f}$. In the following, we shall  identify $M$ with $\delta(M)$ in $H(C_b^k(M))$.
\end{itemize}

The next proposition can be proved in a similar way to the Riemannian case \cite{GaJaRa}.
\begin{prop} \label{countable:neigh}
Let $M$ be a complete  $C^{k}$ Finsler manifold that is $C^k$ uniformly bumpable. Then, $\varphi \in H(C_b^k(M))$ has a countable neighborhood basis in $H(C_b^k(M))$ if and only if $\varphi\in M$.
\end{prop}

\section{A Myers-Nakai Theorem}

Our main result is the following Banach-Stone type theorem for a certain class of Finsler manifolds. It  states that the algebra structure of $C_b^k(M)$ determines the $C^k$ Finsler manifold. Recall that two normed algebras $(A,||\cdot||_A)$ and $(B,||\cdot||_B)$ are \emph{equivalent as normed algebras} whenever there exists an algebra isomorphism $T:A\to B$  satisfying $||T(a)||_B=||a||_A$ for every $a\in A$.
Let us begin by defining the class of Banach spaces where the Finsler manifolds shall be modeled.

%%%%%%%%%%%%%%%
%%%%%%%%%%%%%%%%%
%%%%%%%%%%%%%%%%%%%%%

\begin{defn} A Banach space $(X,||\cdot||)$ is said to be \textbf{k-admissible} if for every equivalent norm $|\cdot|$ and $\varepsilon>0$, there are an open subset $B\supset \{x\in X: |x|\le 1\}$ of $X$ and  a $C^k$ smooth function $g:B\to \Real$ such that
\begin{enumerate}
\item[(i)] $|g(x)-|x||<\varepsilon$ for $x\in B$, and
\item[(ii)] $\Lip(g)\le (1+\varepsilon)$ for the norm $|\cdot|$.
\end{enumerate}
\end{defn}

It is easy to prove the following lemma.

\begin{lem} Let $X$ be a Banach space with one of the following properties:
\begin{enumerate}
\item[(A.1)] Density of the set of equivalent $C^k$ smooth norms: every equivalent norm on $X$ can be approximated in the Hausdorff metric
by equivalent $C^k$ smooth norms \cite{DGZ}.
\item[(A.2)]  $C^k$-fine approximation property ($k\ge 2$) and density of the set of equivalent $C^1$ smooth norms:
For every $C^1$ smooth function $f:X\rightarrow \mathbb R$ and every $\varepsilon>0$,
there is a $C^k$ smooth function $g:X\rightarrow \mathbb R$ satisfying $|f(x)-g(x)|<\varepsilon$  and  $||f'(x)-g'(x)||<\varepsilon$
for all $x\in X$
(see \cite{HJ}, \cite{Azfrygiljarlovo} and \cite{Moulis}); also, every equivalent norm defined on $X$ can be approximated in the Hausdorff metric
by equivalent $C^1$ smooth norms (see \cite[Theorem II 4.1]{DGZ}).
\end{enumerate}
Then $X$ is k-admissible.
\end{lem}

Banach spaces satisfying condition (A.2) are, for instance, separable Banach spaces with a Lipschitz $C^k$ smooth bump function.
  Banach spaces satisfying condition (A.1) for $k=1$ are, for instance, Weakly Compactly Generated (WCG) Banach spaces  with a $C^1$ smooth bump function.

\begin{thm}\label{Myers:Nakai}
Let $M$ and $N$ be complete $C^k$ Finsler manifolds that are $C^k$ uniformly bumpable  and are modeled on $k$-admissible  Banach spaces.
Then $M$ and $N$ are weakly $C^k$ equivalent as Finsler manifolds if and only if  $C_b^k(M)$ and $C_b^k(N)$ are equivalent as normed algebras. Moreover, every normed algebra isomorphism $T:C_b^k(N)\to C_b^k(M)$ is of the form $T(f)=f\circ h$ where $h:M\to N$ is a weak $C^k$ Finsler isometry. In particular,  $h$ is a $C^{k-1}$ Finsler isometry whenever $k\ge 2$.
\end{thm}

In order to prove Theorem \ref{Myers:Nakai}, we shall follow the ideas of the Riemmanian case \cite{GaJaRa}. Let us  divide the proof into several propositions.

\begin{prop}\label{isom}
Let $M$ and $N$ be $C^k$ Finsler manifolds  such that $N$ is %$C^k$ uniformly bumpable ($1\le k \le  k$) and it is
modeled on  a $k$-admissible Banach space $Y$. Let $h:M\to N$ be a map such that $T:C_b^k(N)\to C_b^k(M)$ given by   $T(f)=f\circ h$ is continuous. Then $h$ is $||T||$-Lispchitz for the Finsler metrics.
\end{prop}
\begin{proof}

 For every $y\in N$, let us take a chart $\psi_y: V_y\to  Y$ with $\psi_y(y)=0$. Let us consider the equivalent norm on $Y$,  $|||\cdot|||_y:=||d\psi_y^{-1}(0)(\cdot)||_y$ and  fix $\varepsilon>0$. Let us define the ball $B_{|||\cdot|||_y}(z,t):=\{w\in Y: \,
 |||w-z|||_y<t\}$.

\medskip

\noindent \textbf{Fact.} For every  $r>0$  such that  $ B_{|||\cdot|||_y}(0,r)\subset \psi_y(V_y)$ and every $\widetilde{\varepsilon}>0$, there exists a $C^k$ smooth and Lipschitz function $f_y:Y\toÊ\Real$ such that
\begin{enumerate}
\item $f_y(0)=r$,
\item $||f_y||_\infty:=\sup \{|f_y(z)|: \,z\in Y\} = r$,
\item $\Lip(f_y)\le (1+\varepsilon)^2$ for the norm $|||\cdot|||_y$,
\item $f_y(z)=0$ for every $z \in Y$ with $|||z|||_y\ge r$, and
\item $|||z|||_y\le r-f_y(z)+\widetilde{\varepsilon}$ for every $|||z|||_y\le {r}$.
\end{enumerate}
Let us prove the Fact. First of all, let us take $r>0$, $\widetilde{\varepsilon}>0$ and $0<\alpha< \min\{1,\frac{\varepsilon}{4},\frac{2\widetilde{\varepsilon}}{5r}\}$. Since $N$ is a %$C^k$ uniformly bumpable,
$C^{k}$ Finsler manifold  modeled on a $k$-admissible Banach space $Y$, there are an open subset 
$B\supset \{x\in Y: |||x|||_y\le  1\}$ of $Y$ and a $C^k$ smooth function $g:B\to \Real$ such that
%%%
\begin{itemize}
\item[(i)] $|g(x)-|||x|||_y|<\alpha/2$ on $B$, and
\item[(ii)] $\Lip(g)\le (1+\alpha/2)$ for the norm $|||\cdot|||_y$.
\end{itemize}
%%%
Now, let us take a $C^\infty$ smooth and Lipschitz function $\theta:\Real \to [0,1]$ such that
%%%%
\begin{itemize}
\item[(i)] $\theta(t)=0$ whenever   $t\le \alpha$,
\item[(ii)] $\theta(t)=1$ whenever $t\ge 1-\alpha$,
\item[(iii)] $\Lip(\theta)\le (1+\varepsilon)$, and
\item[(iv)] $|\theta(t)-t|\le 2\alpha$ for every $t\in[0,1+\alpha]$.
\end{itemize}
%%%%
Let us define 
\begin{equation*}
f(x)=
\begin{cases}
\theta(g(x)) & \text{if $x\in B$},\\
1 & \text{if $x\in Y\setminus B.$}
\end{cases}
\end{equation*}
It is straightforward to verify that $f$ is well-defined,  $C^k$ smooth, $f(x)=1$ whenever $|||x|||_y \ge 1$ and $f(x)=0$ whenever
$|||x|||_y \le \alpha/2$ 
%%%%%%%%%%
%%%%%%%%
Let us now consider  $f_y:Y\to [0,r]$ as $f_y(z)=r(1-f(\frac{z}{r}))$, which is $C^k$ smooth, Lipschitz and satisfies:
\begin{itemize}
\item[(i)] $f_y(0)=r$,
\item[(ii)] $||f_y||_\infty = r$,
\item[(iii)] $|f_y(z)-f_y(x)|\le (1+\varepsilon)(1+\alpha/2)|||z-x|||_y\le (1+\varepsilon)^2|||z-x|||_y$,
\item[(iv)] $f_y(z)=0$ for every $z\in Y$ with $|||z|||_y\ge r$,
\item[(v)] $|||\frac{z}{r}|||_y\le \frac{\alpha}{2} + g(\frac{z}{r})\le \frac{\alpha}{2}+2\alpha+f(\frac{z}{r})$ for every $|||z|||_y\le r$. Thus,
$|||z|||_y\le r(\frac{\alpha}{2}+2\alpha) +r-f_y(z)\le \widetilde{\varepsilon}+r-f_y(z)$ for every $|||z|||_y\le {r}$.
\end{itemize}
%%%%%%

\medskip

%%%%%%
Let us now prove Proposition \ref{isom}. Let us fix $p_1,p_2\in M$ and $\varepsilon>0$. Let us consider $\sigma:[0,1]\to M$ a piecewise  $C^1$ smooth path in $M$ joining $p_1$ and $p_2$, with $\ell(\sigma)\le d_M(p_1,p_2)+\varepsilon$. Since $h:M\to N$ is continuous, the path $\widehat{\sigma}:=h\circ \sigma: [0,1]\to N$,  joining $h(p_1)$ and $h(p_2)$,
is continuous as well.  For every $q\in \widehat{\sigma}([0,1])$, there is $0<r_q<1$ and a chart $\psi_q:V_q\to Y$ such that $\psi_q(q)=0$, $B_N(q,r_{q})\subset V_q$ and the bijection $\psi_q:V_q\to \psi_q(V_q)$ is $(1+\varepsilon)$-bi-Lipschitz for the norm $||d\psi^{-1}_q(0)(\cdot)||_q$ in $Y$
(see Lemma \ref{desigualdades:BiLipschitz}). Since  $\widehat{\sigma}([0,1])$ is a compact set of $N$, there is a finite family
of points  $0=t_1<t_2<...<t_m=1$ and a family of open intervals $\{I_{k}\}_{k=1}^m$  covering the interval $[0,1]$ so that,
if we define  $q_k:=\widehat{\sigma}(t_k)$ and $r_k:=r_{q_k}$, for every $k=1,...,m$, we have
\begin{enumerate}
\item[(a)] $\widehat{\sigma}(I_{k})\subset B_N(q_k, r_{k}/(1+\varepsilon))$,
\item[(b)] $I_{j}\cap I_{k}\neq \emptyset$ if, and only if, $|j-k|\le 1$.
\end{enumerate}
It is clear that $\widehat{\sigma}([0,1])\subset \bigcup_{k=1}^{m} B_N(q_k,\frac{r_k}{1+\varepsilon})$. Now,  let us select a point $s_k\in I_k\cap I_{k+1}$ such that $t_k<s_k<t_{k+1}$,  for every $k=1,...,m-1$. Let us write $a_k:=\widehat{\sigma}(s_k)$, for every
$k=1,\cdots, m-1$,
 $\psi_k:=\psi_{q_k}$, $V_k:=V_{q_k}$ and $|||\cdot|||_k:=||d\psi^{-1}_{k}(0)(\cdot)||_{q_k}$, for every $k=1,....m$.
Notice that $a_k\in  B_N(q_k, \frac{r_k}{1+\varepsilon})\cap B_N(q_{k+1}, \frac{r_{k+1}}{1+\varepsilon})$, for every
$k=1, \cdots, m-1$.
 Since $\psi_{k}:V_{k}\to \psi_{k}(V_{k})$ is $(1+\varepsilon)$-bi-Lipschitz for the norm $|||\cdot|||_{k}$ in $Y$,
 we deduce that
   $\psi_{k}(a_k)\in B_{|||\cdot|||_{k}}(0,r_{k})$, for every
$k=1, \cdots, m-1$.

Now, let us we apply  the above Fact to $r_k$, $\varepsilon$ and $\widetilde{\varepsilon}=\varepsilon/2m$ to obtain  functions  $f_{k}:Y\to [0,r_k]$ satisfying   properties (1)--(5), $k=1,\cdots, m$. Let us  define the $C^k$ smooth and Lipschitz functions $g_k:N \to [0,r_k]$ as $g_k(z)=f_{k}(\psi_k(z))$ for every $z\in V_k$ and $g_k(z)=0$ for $z\not\in V_k$, $k=1,\cdots,m$. Then,
\begin{enumerate}
\item[(i)] $g_k\in C_b^k(N)$;
\item[(ii)] $g_k(q_k)=r_{k}$;
\item[(iii)] $|g_k(z)-g_k(x)|\le (1+\varepsilon)^3 d_N(z,x)$ for all $z,x\in N$;
\item[(iv)] If $z\in \psi_k^{-1}(B_{|||\cdot|||_k}(0,{r_k}))$, then
 $|||\psi_k(z)|||_{k}\le {r_k}$ and from condition (5) on the Fact, we obtain
\begin{equation*}
d_N(z,q_k)\le (1+\varepsilon)|||\psi_k(z)-\psi_k(q_k)|||_{k}  = (1+\varepsilon)|||\psi_k(z)|||_{k} \le (1+\varepsilon)(r_{k}-g_k(z)+\varepsilon/2m).
\end{equation*}
\end{enumerate}
The Lipschitz constant of $g_k\circ h$, for $k=1,\cdots, m$, is the following
\begin{align*}
\Lip(g_k\circ h) & \le ||g_k\circ h||_{C_b^1(M)}=||T(g_k)||_{C_b^1(M)} \le ||T|| ||g_k||_{C_b^1(N)}= \\
& = ||T|| \max\{||g_k||_\infty, ||dg_k||_\infty\}\le ||T|| (1+\varepsilon)^3.
\end{align*}
Now, since  $r_{k}=g_{k}(q_k)=g_{k}(h(\sigma(t_k)))$ and $\psi_{k}(h(\sigma(s_k)))\in B_{|||\cdot|||_{k}}(0,{r_k})$, we have
%%%%%%%%%%%%%%%%%%
\begin{align*}
d_N(h(p_1),&h(p_2))   \le \sum_{k=1}^{m-1} [d_N(h(\sigma(t_k)), h(\sigma(s_k))) +d_N(h(\sigma(s_k)), h(\sigma(t_{k+1}))) ] \le \\
&  \le \sum_{k=1}^{m-1} (1+\varepsilon)[ g_{k}(q_k)-g_{k}(h(\sigma(s_k)))+\\
&\qquad \qquad \quad \qquad \qquad \qquad \qquad + g_{k+1}(q_{k+1})-g_{k+1}(h(\sigma(s_k))) + \varepsilon/m ] \le \\
& \le \sum_{k=1}^{m-1} (1+\varepsilon)[\Lip(g_{k}\circ h) d_M(\sigma(t_k),\sigma(s_k)) + \\
& \qquad \qquad\qquad\qquad \qquad \qquad + \Lip(g_{k+1}\circ h)d_M(\sigma(t_{k+1}),\sigma(s_k)) +\varepsilon/m] \le \\
& \le \sum_{k=1}^{m-1} ||T||(1+\varepsilon)^4[d_M(\sigma(t_k),\sigma(s_k)) + d_M(\sigma(t_{k+1}),\sigma(s_k))] +\varepsilon(1+\varepsilon) \le \\
& \le \sum_{k=1}^{m-1} ||T||(1+\varepsilon)^4 \ell(\sigma_{\mid_{[t_k,t_{k+1}]}})+\varepsilon(1+\varepsilon) = ||T||(1+\varepsilon)^4 \ell(\sigma) +\varepsilon(1+\varepsilon) \le \\
&  \qquad \qquad\qquad \qquad \qquad  \qquad \qquad \le  ||T||(1+\varepsilon)^4(d_M(p_1,p_2)+\varepsilon)+\varepsilon(1+\varepsilon)
\end{align*}
for every $\varepsilon>0$. Thus, $h$ is $||T||$-Lipschitz.
\end{proof}

%%%%%%%%%%%%%%%%%%
%%%%%%%%%%%%%%%%%%%%
%%%%%%%%%%%%%%%%%%

\begin{lem}\label{weak:smooth}
Let $M$ and $N$ be $C^k$ Finsler manifolds such that $N$ is modeled on a Banach space with a Lipschitz $C^k$ smooth bump function.
Let $h:M\to N$ be a homeomorphism such that $f\circ h \in C_b^k(M)$ for every $f\in C_b^k(N)$. Then, $h$ is a weakly $C^k$ smooth  function on $M$.
\end{lem}
\begin{proof}
Let us fix $x\in M$ and $\varepsilon=1$. There are charts  $\varphi:U\to X$ of $M$ at $x$ and $\psi: V\to Y$ of $N$ at $h(x)$ satisfying inequalities  \eqref{palaisdef} and \eqref{B-L1} on $U$ and $V$, respectively. We can assume that  $h(U)\subset V$. Since $Y$ admits  a Lipschitz and $C^k$ smooth bump function and $\psi(h(U))$ is an open neighborhood of $\psi(h(x))$ in $Y$, there are real numbers $0<s<r$ such that $B(\psi(h(x)),s) \subset B(\psi(h(x)),r) \subset \psi(h(U))$ and a Lipschitz and  $C^k$ smooth function $\alpha:Y\to \Real$ such that $\alpha(y)=1$ for  $y\in B(\psi(h(x)),s)$ and $\alpha(y)=0$ for $y\not\in B(\psi(h(x)),r)$. Let us define $U_0:=h^{-1}(\psi^{-1}(B(\psi(h(x)),s)))\subset U$, which is an open neighborhood of $x$ in $M$.

%Since $N$ is $C^k$ uniformly bumpable, there are $R>1$ and $r>0$ such that for every  $\delta\in (0,r)$,
%$B_N(h(x),\delta)\subset h(U)$ and there exists a $C^k$ smooth function $b:M\to[0,1]$ such that $b(h(x))=1$, $b(y)=0$ whenever $d_N(y,h(x))\ge \delta$, and  $\Lip(b)\le \sup_{y\in N} ||d b(y)||_y \le R/\delta$. Thus, $b(y)\ge 1- \Lip(b) d_N(y, h(x))\ge 1- \frac{R}{\delta}d_N(y,h(x))$ for every $y\in N$. Therefore, if we take $y\in  V_0:=B_N(h(x),\frac{\delta}{2R})\subset V$, then $b(y)\ge 1-1/2=1/2$. Let us take a Lipschitz and $C^\infty$ smooth function $\theta:\Real\to [0,1]$ such that $\theta(t)=0$ for every $t\le 0$ and $\theta(t)=1$ for every $t\ge 1/2$. Let us define $\alpha:N\to\Real$ as $\alpha(y)=\theta(b(y))$. Clearly, $\alpha \in C_b^k(N)$.  Moreover, $\alpha(y)=1$ whenever $y\in V_0$ and $\alpha(y)=0$ whenever $y\not\in B_N(h(x),\delta)$. Let us define $U_0:=h^{-1}(V_0)\subset U$, which is an open neighborhood of $x$ in $M$.

Let us check that  $y^* \circ (\psi \circ  h \circ \varphi^{-1})$ is $C^k$ smooth on $\varphi(U_0)\subset X$ for all
$y^* \in Y^*$. Following the proof of  \cite[Theorem 4]{GAJAANG}, we define $g:N\to \Real$ as $g(y)=0$ whenever $y\not\in V$ and $g(y)=\alpha(\psi(y))\cdot y^*(\psi(y))$ whenever $y\in V$. It is clear that $g\in C_b^k(N)$ and, by assumption, $g\circ h \in C_b^k(M)$. Now, it follows that $\psi(h(\varphi^{-1}(z)))\in B(\psi(h(x)),s)$  for every $z\in \varphi(U_0)$. Thus
\begin{align*}
y^*\circ (\psi \circ  h \circ \varphi^{-1})(z) & = y^* (\psi(h(\varphi^{-1}(z))))=\alpha(\psi(h(\varphi^{-1}(z)))) y^*(\psi(h(\varphi^{-1}(z))))= \\
& =g(h(\varphi^{-1}(z)))=g\circ h\circ \varphi^{-1} (z),
\end{align*}
for every $z\in \varphi(U_0)$.
Since $(g\circ h)\circ \varphi^{-1}$ is $C^k$ smooth on $\varphi(U_0)$, we have that $y^*\circ (\psi \circ  h \circ \varphi^{-1})$ is $C^k$ smooth on $\varphi(U_0)$. Thus $\psi \circ  h \circ \varphi^{-1}$ is weakly $C^k$ smooth on $\varphi(U_0)$ and $h$ is weakly $C^k$ smooth on $M$.
\end{proof}

\medskip

\begin{prooftheoMyersNakai}
If $h:M\to N$ is a weak $C^k$ Finsler isometry,  we can define the  operator $T:C_b^k(N)\to C_b^k(M)$  by $T(f)=f\circ h$.
Let us check that $T$ is well defined. For every $x\in M$, there are charts $\varphi:U\to X$ of $M$ at $x$ and $\psi:V\to Y$ of $N$ at $h(x)$, such that $h(U)\subset V$ and $\psi\circ h \circ \varphi^{-1}$ is weakly $C^k$ smooth on $\varphi(U)\subset X$.
Also, $f\circ \psi^{-1}$  is $C^k$ smooth on $\psi(V)\subset Y$.
Thus,  by \cite[Proposition 4.2]{GULLA}, $(f\circ \psi^{-1})\circ (\psi \circ h \circ \varphi^{-1})=f \circ h \circ \varphi^{-1}$ is $C^k$ smooth on $\varphi(U)$.
Therefore, $f\circ h $ is $C^k$ smooth on $U$. Since this holds for every $x\in M$, we deduce that $f\circ h $ is $C^k$ smooth on  $M$. Moreover, $T$ is an algebra isomorphism with $||T(f)||_{C_b^1(M)}=||f\circ h||_{C_b^1(M)}=||f||_{C_b^1(N)}$ for every $f\in C_b^k(N)$.

\medskip

Conversely, let $T: C_b^k(N)\to C_b^k(M)$  be a normed algebra isometry. Then, we can define the function $h:H(C_b^k(M))\to H(C_b^k(N))$ by $h(\varphi)=\varphi\circ T$ for every $\varphi \in H(C_b^k(M))$. The function $h$ is a bijection. Moreover, $h$  is an homeomorphism.
% since $\pi_f\circ h = \pi_{T(f)}$ for every $f\in C_b^1(N)$. Indeed,  $\pi_f\circ h (\varphi)=\pi_f ( h(\varphi)) =\pi_f (\varphi\circ T) = \varphi (T(f))=\pi_{T(f)}(\varphi) $  for every $f\in C_b^1(N)$  and $\varphi\in H(C_b^1(M))$, then $h$ in injective, since if $h(\varphi)=h(\varphi')$, then $\varphi(T(f))=\pi_f\circ h (\varphi)=\pi_f\circ h(\varphi')=\varphi'(T(f))$ for every $f\in C_b^1(N)$, but $T$ is an isomorphism, then $\varphi=\varphi'$. The function $h$ is a bijection, since $h^{-1}(\varphi)=\varphi\circ T^{-1}$ for every $\varphi\in H(C_b^1(N))$. Let us see that $h:H(C_b^1(M))\to H(C_b^1(N))$ is continuous. Fix $\varphi\in H(C_b^1(M))$ and $V\subset H(C_b^1(N))$ an basic neighborhood of $h(\varphi)$ in $H(C_b^1(N))$, i.e. $V=\{\psi\in H(C_b^1(N)) : |\psi(f_i)-h(\varphi)(f_i)|<\varepsilon,\ i=1,...,n \}$ with $f_1,...,f_n \in C_b^1(N)$ and $\varepsilon>0$. We can define the following neighborhood of $\varphi$ in $H(C_b^1(M))$ (note that $T(f_i)\in C_b^1(M)$):
%\begin{equation*}
%U=\{\phi\in H(C_b^1(M)): |\phi(T(f_i))-\varphi(T(f_i))|<\varepsilon,\ i=1,...n \},
%\end{equation*}
%since $h(\phi)=\phi\circ T$ we have that $h(U)=V$.
Recall that we  identify  $x\in M$ with $\delta_x\in C_b^k(M)$. Thus, $h(x)=h(\delta_x)=\delta_x\circ T$.
Since $h$ is an homeomorphism,  by Proposition \ref{countable:neigh}, we obtain  for every $p\in N$ a unique
point $x\in M$ such that
$h(\delta_x)=\delta_{p}$. Let us check that  $T(f)=f\circ h$ for all $f\in C_b^k(N)$. Indeed,  for every $x\in M$ and every $f\in C_b^k(N)$,
\begin{equation*}
T(f)(x)=\delta_x(T(f))=(\delta_x\circ T)(f)=h(\delta_x)(f)=\delta_{h(x)}(f)=f(h(x))=f\circ h (x).
\end{equation*}

Now, from Proposition \ref{isom} and Lemma \ref{weak:smooth} we deduce that $h$ is a weak $C^k$ Finsler isometry.
\end{prooftheoMyersNakai}

\begin{rem}
It is worth mentioning that, for  Riemannian manifolds, every metric isometry is a $C^1$ Finsler isometry. This result was proved
by S. Myers and N. Steenrod \cite{MyStee} in the finite-dimensional case and by I. Garrido, J.A. Jaramillo and Y.C. Rangel \cite{GaJaRa}  in the general case. Also, S. Deng and Z. Hou \cite{DengHou} obtained a version for  finite-dimensional  Riemannian-Finsler manifolds. Nevertheless, there is no a generalization, up to our knowledge, of the  Myers-Steenrod theorem  for all Finsler manifolds. Thus, for $k=1$ we can only assure that the  metric isometry obtained in Theorem \ref{Myers:Nakai}  is weakly $C^1$ smooth.
\end{rem}

Let us finish this note with some interesting corollaries of Theorem \ref{Myers:Nakai}. First, recall that every separable Banach space with a Lipschitz $C^k$ smooth bump function satisfies  condition (A.2) and every WCG Banach space with a $C^1$ smooth bump function satisfies  condition (A.1) for $k=1$. 
\begin{cor}
Let $M$ and $N$ be complete,  $C^{1}$ Finsler manifolds that are $C^1$ uniformly bumpable and are modeled on WCG Banach spaces. Then $M$ and $N$ are weakly $C^1$ equivalent as Finsler manifolds if, and only if, $C_b^1(M)$ and $C_b^1(N)$ are equivalent as normed algebras. Moreover, every normed algebra isomorphism $T:C_b^1(N)\to C_b^1(M)$ is of the form $T(f)=f\circ h$ where $h:M\to N$ is a weak $C^1$ Finsler isometry.
\end{cor}

\begin{cor}
Let $M$ and $N$ be complete, separable $C^{k}$ Finsler manifolds that are modeled on Banach spaces with a Lipschitz and $C^k$ smooth bump function.
Then $M$ and $N$ are weakly $C^k$ equivalent as Finsler manifolds if and only if $C_b^k(M)$ and $C_b^k(N)$ are equivalent as normed algebras. Moreover, every normed algebra isomorphism $T:C_b^k(N)\to C_b^k(M)$ is of the form $T(f)=f\circ h$ where $h:M\to N$ is a weak $C^k$ Finsler isometry. In particular, $h$ is a $C^{k-1} $ Finsler isometry whenever $k\ge 2$.
\end{cor}

Since every   weakly $C^k$ smooth function with values in a finite-dimensional normed space is $C^k$ smooth and every finite-dimensional
$C^{k}$ Finsler manifold is $C^k$ uniformly bumpable \cite{MarLuis}, we obtain the following Myers-Nakai result for finite-dimensional $C^{k}$ Finsler manifolds.

\begin{cor}
Let $M$ and $N$ be complete and finite dimensional $C^{k}$ Finsler manifolds. Then $M$ and $N$ are $C^k$ equivalent as Finsler manifolds if, and only if, $C_b^k(M)$ and $C_b^k(N)$ are equivalent as normed algebras. Moreover, every normed algebra isomorphism $T:C_b^k(N)\to C_b^k(M)$ is of the form $T(f)=f\circ h$ where $h:M\to N$ is a $C^k$ Finsler isometry.
\end{cor}

We obtain an interesting application of Finsler manifolds to Banach spaces. Recall the well-known Mazur-Ulam Theorem establishing that every surjective isometry between two Banach spaces is affine.
\begin{cor}
Let $X$ and $Y$ be WCG Banach spaces with  $C^1$ smooth bump functions. Then $X$ and $Y$ are isometric if, and only if, $C_b^1(X)$ and $C_b^1(Y)$ are equivalent as normed algebras. Moreover, every normed algebra isomorphism $T:C_b^1(Y)\to C_b^1(X)$ is of the form $T(f)=f\circ h$ where $h:X\to Y$ is a surjective isometry. In particular, $h$ and $h^{-1}$ are  affine isometries.
\end{cor}

%Finally, let us point out that we recover the Myers-Nakai theorem for Riemannian manifolds given in \cite{GaJaRa}. Recall that every separable Riemannian manifold is $C^\infty$ uniformly bumpable \cite{AzFeMeRa} and every non-separable Riemannian manifold is  $C^1$ uniformly bumpable \cite{MarLuis}.

%\begin{cor}\cite{GaJaRa}
%Let $M$ and $N$ be complete Riemannian manifolds. Then $M$ and $N$ are equivalent as Riemannian manifolds if, and only if, $C_b^1(M)$ and $C_b^1(N)$ are equivalent as normed algebras. Moreover, every normed algebra isomorphism $T:C_b^1(N)\to C_b^1(M)$ is of the form $T(f)=f\circ h$ where $h:M\to N$ is a Riemannian isometry.
%\end{cor}

%%%%%%%%%
%%%%%%%%%%%%
%%%%%%%%%%%%%%%

\end{document}